\documentclass[12pt]{article}
\usepackage[english]{babel}
\usepackage{amsthm, amssymb,amsmath,latexsym}

\newtheorem{theorem}{Theorem}

\newtheorem{lemma}{Lemma}

\numberwithin{equation}{section}
\begin{document}

\title{On Bounds for the Smallest and the Largest Eigenvalues of GCD and LCM Matrices}

\author{Ercan Alt\i n\i \c{s}\i k  \\
\small ealtinisik@gazi.edu.tr, ealtinisik@gmail.com\\
\small Department of Mathematics, Faculty of Sciences, Gazi University \\
\small 06500 Teknikokullar - Ankara, Turkey
\and
 \c{S}erife B\"{u}y\"{u}kk\"{o}se \\
\small sbuyukkose@gazi.edu.tr, serifebuyukkose@gmail.com\\
\small Department of Mathematics, Faculty of Sciences, Gazi University \\
\small 06500 Teknikokullar - Ankara, Turkey
}
\date{\vspace{-5ex}}
\maketitle

% ----------------------------------------------------------------
\begin{abstract}
In this paper, we study the eigenvalues of the GCD matrix $(S_n)$ and the LCM matrix $[S_n]$ defined on $S_n=\{1,2,\ldots,n\}$. We present upper and lower bounds for the smallest and the largest eigenvalues of $(S_n)$ and $[S_n]$ in terms of particular arithmetical functions.



\end{abstract}

\section{Introduction and Preliminaries}
Let $S=\{x_1,x_2,\ldots,x_n\}$ be a set of distinct positive integers.
$(x_i,x_j)$ and $[x_i,x_j]$ denote the greatest common divisor and the least
common multiply of $x_i$ and $x_j$, respectively. The $n \times n$ matrices
$(S)=((x_i,x_j))$ and $[S]=([x_i,x_j])$ are respectively called the GCD matrix and the LCM matrix on $S$. Initially, Smith \cite{Smith} proved that
$det(S)=\prod_{k=1}^n \varphi(k)$, where $\varphi$ is Euler's totient. Since
Smith's paper many generalizations of Smith's result have been published in the
literature. For general accounts see e.g.
\cite{{Altinisik2005},{Hauk1997},{KorkeeHauk2003},{SandorCrs}}.

In 1989, Beslin and Ligh \cite{BeslinLigh} proved that $(S)$ is positive definite for any set $S$ of distinct positive integers but $[S]$ is not positive definite. Therefore, all eigenvalues of
$(S)$ are positive reals but all eigenvalues of $[S]$ need not be positive (see \cite{BourqueLigh1992}).

On the other hand, some authors studied the eigenvalues of GCD-related matrices by using
some tools of functional analysis \cite{Wintner, Lindqvist} Let $\lambda_n^{(1)} \leq \lambda_n^{(2)}
\leq \cdots \leq \lambda_n^{(n)}$ be the eigenvalues of the $n \times$ matrix $
(M_n^\varepsilon)=\big( \frac {(i,j)^{2\varepsilon}}{i^\varepsilon
j^\varepsilon} \big), $ where $\varepsilon$ is a real number.

In 2004, Hong and Loewy \cite{Hong} have investigated the asymptotic behavior
of the eigenvalues of power GCD matrices. Let $ \{x_i\}_{i=1}^\infty $ be a
given arbitrary strictly increasing infinite sequence of positive integers. For
any integer $n \geq 1$, let $S_n = \{ x_1,x_2,\ldots,x_n \}$ and $\varepsilon$
be a real number. The $n \times n$ matrix
$(S_n^\varepsilon)=((x_i,x_j)^\varepsilon)$ is called the power GCD matrix on
$S_n$. Let $\lambda_n^{(1)} \leq \cdots \leq \lambda_n^{(n)}$ be the
eigenvalues of the power GCD matrix $(S_n^\varepsilon)$. For a sequence
$\{x_i\}_{i=1}^\infty$ satisfying that for any $i \neq j$, $(x_i,x_j)=x$, a
given integer and that $\sum _{i=1}^\infty \frac{1} {x_i}=\infty$, they showed
that $\lim_{n\rightarrow\infty} \lambda_n^{(1)}=x_1^\varepsilon-x^\varepsilon$
if $0 < \varepsilon \leq 1$. Then, for the arithmetic progression $\{
x_{i-e+1}=a+bi \}_{i=1}^\infty $, where $a \geq 0$, $b \geq 1$ and $e \geq 0$
are any given integers, they showed that if $0 < \varepsilon \leq 1 $, then
$\lim_{n\rightarrow\infty} \lambda_n^{(q)}=0$ for a fixed integer $q \geq$.
They also conjectured that if $\epsilon>1$ then $\lim_{n\rightarrow \infty}
\lambda_n ^{(1)} >0$.

In 2009, Alt\i n\i \c{s}\i k \cite{Altinisik2009} investigated the inverse of
GCD matrices associated with multiplicative functions and gave a proof of the
Hong-Loewy conjecture. Let $C=\{1,2,\ldots,m\}$ and $f$ be a multiplicative
function such that $(f \ast \mu)(k) > 0$ for every positive integer $k$ and the
Euler product $\zeta_f = \prod_\wp (1-\frac{1}{f(\wp)})^{-1}$ converges. Let
$(C_f)=(f(i,j))$ be the $m\times m$ matrix defined on the set $C$ having $f$
evaluated at the greatest common divisor $(i,j)$ of $i$ and $j$ as its
$ij-$entry. Alt\i n\i \c{s}\i k gave lower bounds for the smallest eigenvalue $\lambda_m^{(1)}(C_f)$ of the $m \times m$ matrix $(C_f)$ in terms of the Riemann zeta
function for each $f=N^\varepsilon$, $J_\varepsilon$ and
$\sigma_\varepsilon$. In the same paper, he showed that $\lim_{m \rightarrow
\infty}\lambda_m^{(1)}(C_{N^\varepsilon})>0$ if $\varepsilon > 1$ and gave a proof of the Hong-Loewy conjecture.

In the study of eigenvalues of the GCD and related matrices Hong and Enoch Lee
\cite{HongEnochLee} examined the asymptotic behavior of the eigenvalues of the
reciprocal LCM matrix $[\frac{1}{S^r} ] = (\frac{1}{[x_i,x_j]^r})$ on $S_n = \{
x_1,x_2,\ldots,x_n \}$ for a positive real number $r$. Let $\lambda_n^{(q)}$ be
the $q$-th smallest eigenvalue of $[\frac{1}{S^r} ]$ for a given arbitrary
integer $q \geq
1$. They showed that $ 0 < \lambda_n^{(1)} \leq \frac{1}{n} ( \frac{1}{x_1^r} + \cdots +
\frac{1}{x_n^r} )$, $ \lim_{n \rightarrow \infty} \lambda_n^{(q)} =0 $, and $\lambda_n^{(k)} \leq \frac{k}{x_{n-k+1}^r}$ for $1 \leq k \leq n$. Also, Hong \cite{Hong2008} obtained a lower bound depending only $x_1$ and $n$ for the smallest eigenvalue of $(f(x_i,x_j))$ having evaluated at $(x_i,x_j)$ as its $ij-$entry if $(f\ast\mu)(d)>0$ whenever $d|x$ for $x\in S_n$.

Let $c \geq 1$ and $d \geq 0$ be integers and $q \geq 1$ be a given arbitrary
integer. Let $f_1,f_2, \ldots , f_c \in C = \{f : (f \ast \mu)(d^\prime )\geq 0
\textrm{ whenever } d^\prime | x \textrm{ for any } x \in \{x_i\}_{i=1}^\infty
 \} $ be distinct and $(\ell_1, \ldots , \ell_c) \in \mathbb{Z}_{>0}^c$ satisfy
$\ell_1+ \cdots + \ell_c >d$. Let $\lambda_n^{(1)} \leq \cdots \lambda_n^{(n)}$
be eigenvalues of the $n \times n$ matrix $((f_1^{(\ell_1)} \ast f_2^{(\ell_2)}
\ast \cdots \ast f_c^{(\ell_c)} \ast \mu^{(d)} ) (x_i,x_j) )$ defined on
$S=\{x_1,x_2,\ldots,x_n\}$, where
$$f_k^{(\ell_k)}=\underbrace{f_k \ast \cdots
\ast f_k}_{\ell_k \textrm{ times } f_k } \textrm{  and  }  \mu^{(d)}
=\underbrace{\mu \ast \cdots \ast \mu}_{d \textrm{ times } \mu}. $$ Then Hong and Loewy \cite{HongLoewy2011} showed that the matrix
$$
\big( (f_1^{(\ell_1)} \ast f_2^{(\ell_2)} \ast \cdots \ast f_c^{(\ell_c)} \ast
\mu^{(d)} ) (x_i,x_j) \big)
$$
is positive definite and $ \lim_{n \rightarrow \infty}
\lambda_n^{(q)} \geq 0 $ for a given arbitrary integer $q \geq 1$.

On a different point of view, in 2008, the authors
\cite{IlmonenHaukkanenMerikoski} examined the eigenvalues of certain abstract
generalizations of the GCD matrix and the LCM matrix on posets. Recently,
Mattila and Haukkanen \cite{MattilaHauk2014} and Mattila \cite{Mattila2014} gave new bounds for the eigenvalues of such abstract generalizations.

In addition to above papers, in 2012, Mattila and Haukkanen \cite{Mattila2012, Mattila2012paper} investigated the eigenvalues of the $n \times n$ matrix $A^{\alpha,\beta}_n = ((i,j)^\alpha [i,j]^\beta)$, where $\alpha, \beta \in \mathbb{R}$. They proved that $\lambda_n^{(1)} \geq t_n \cdot \min_{1\leq i \leq n } J_{\alpha-\beta}(i) \cdot \min\{ 1, n^{2 \beta} \}$, where $\lambda_n^{(1)}$ is the smallest eigenvalue of $A^{\alpha,\beta}_n$ and $\alpha > \beta$, and obtained a real interval which provides a broad bounds for the eigenvalues of the matrix $A^{\alpha,\beta}_n$.

Except $A^{\alpha,\beta}_n$, all above matrices of which eigenvalues were investigated are positive
definite. It is clear that the LCM matrix is not positive definite so some
eigenvalues are negative. Therefore, to investigate asymptotic behavior of the eigenvalues of the LCM matrix $[S]$ is not easy. Except the papers \cite{Mattila2012, Mattila2012paper}, the eigenvalues of the LCM matrix have not hitherto been studied and any upper (lower) bounds for the smallest (largest) eigenvalues of $(S_n)$ and $[S_n]$ have not been presented earlier in the literature. In this paper we investigate the smallest and the largest eigenvalues of the GCD matrix $(S_n)$ as well as the LCM matrix $[S_n]$ defined on $S_n=\{ 1,2,\cdots,n \}$. In order to perform this investigation we use some results of the excellent paper \cite{WolkowiczStyan1980} on the eigenvalue inequalities for a complex matrix. We improve the eigenvalue inequalities given in \cite{WolkowiczStyan1980} for our matrices, which are clearly real symmetric, and hence we present upper and lower bounds for the smallest and the largest eigenvalues of $(S_n)$ and $[S_n]$ in terms of particular arithmetical functions.

Now we summarize number theoretical tools used in this paper. The Dirichlet product $f \ast g$ of arithmetical functions $f$ and $g$ is defined as
\begin{eqnarray*}
(f\ast g)(n) = \sum_{d|n} f(d) g(n/d)
\end{eqnarray*}
and the usual product of $f$ and $g$ is defined as $(fg)(n)=f(n)g(n)$ for all positive integers $n$. We denote Euler's totient function by $\varphi$ and the M\"{o}bius function by $\mu$. Also, the power function $N_\alpha$ is defined as $N_\alpha(n)=n^\alpha$ for all positive integers $n$, where $\alpha \in \mathbb{R}$. Note that $N_1=N$. The zeta function $\zeta$ is defined as $\zeta(n)=1$ for all positive integers $n$. One can find undefined terms for arithmetical functions in the text of Sivaramakrishnan \cite{Siva}.

\section{Main Results}
Now we investigate for the smallest and largest eigenvalue of the GCD matrix $(S_n)$ and the LCM matrix $[S_n]$. In order to do this, we use the following elegant result given for any $n \times n$ complex matrix with real eigenvalues. For undefined terms and fundamental results on eigenvalues of matrices one can consult the text of Horn and Johnson \cite{HornJohnson}.

\begin{theorem} [Theorem 2.1 in \cite{WolkowiczStyan1980}] \label{thmWolStyan}
Let $A$ be an $n \times n$ complex matrix with real eigenvalues $ \lambda(A)$,
and let $m=trA/n$, $s^2=tr(A^2)/n - m^2$. Then
\begin{equation} \label{min}
m-s(n-1)^{1/2} \leq \lambda_{min}(A) \leq m-s/(n-1)^{1/2},
\end{equation}
\begin{equation} \label{max}
m+s/(n-1)^{1/2} \leq \lambda_{max}(A) \leq m+s(n-1)^{1/2}.
\end{equation}
Equality holds on the left (right) of (\ref{min}) if and only if equality holds
on the left (right) of (\ref{max}) if and only if the $n-1$ largest (smallest)
eigenvalues are equal.
\end{theorem}

In \cite{WolkowiczStyan1980} Wolkowicz and Styan proved Theorem~\ref{thmWolStyan} by using the following lemmas. In this study, we reprove Lemma~\ref{lemma2.2.WolStyan} for the gcd matrix and also the lcm matrix and hence we improve the upper bound for $\lambda_{min}(A)$ in (\ref{min}) and the lower bound for $\lambda_{max}(A)$ in (\ref{max}) for our matrices.

\begin{lemma}[Lemma 2.1 in \cite{WolkowiczStyan1980}] \label{lemma2.1.WolStyan}
Let $w$ and $\lambda$ be real nonnull $n \times 1$ vectors, and let
\begin{eqnarray}\label{ms}
m=\lambda^{T} e/n \textit{    and     } s^2=\lambda^{T} C \lambda / n,
\end{eqnarray}
where $e$ is the $n \times 1 $ vector of ones, the centering matrix $C=I-ee^{T}$, and $e^{T}$ is the transpose of $e$. Then
\begin{eqnarray}\label{lemma2.1.bounds}
-s (nW^{T}CW)^{1/2} \leq W^{T} \lambda -m W^{T} e = W^{T}C\lambda \leq s(nW^{T}CW)^{1/2}.
\end{eqnarray}
Equality holds on the left (right) of (\ref{lemma2.1.bounds}) if and only if $ \lambda = aw+be $ for some scalars $a$ and $b$, where $a<0$ ($a>0$).
\end{lemma}

It should be noted that $m$ and $s^2$ defined in Theorem~\ref{thmWolStyan} and Lemma~\ref{lemma2.1.WolStyan} are equivalent \cite{WolkowiczStyan1980}. In Lemma~\ref{lemma2.2.WolStyan} we use a convenient notation for components of $\lambda$ as in the study of eigenvalues of gcd and lcm matrices.

\begin{lemma}[Lemma 2.2 in \cite{WolkowiczStyan1980}] \label{lemma2.2.WolStyan}
Let $\lambda =(\lambda_n^{(1)}, \lambda_n^{(2)}, \ldots ,\lambda_n^{(n)})$, $m$ and $s$ be defined as in Lemma~\ref{lemma2.1.WolStyan}, and $\lambda_n^{(1)} \leq \lambda_n^{(2)} \leq \cdots \leq \lambda_n^{(n)}.$ Then
\begin{eqnarray}\label{lemma2.2.eigenvalbounds}
\lambda_n^{(1)} \leq m-\frac{s}{(n-1)^{1/2}} \leq m+\frac{s}{(n-1)^{1/2}} \leq \lambda_n^{(n)}.
\end{eqnarray}
Equality holds on the left if and only if $\lambda_n^{(1)}=\lambda_n^{(2)}=\cdots =\lambda_n^{(n-1)}$, on the right if and only if $\lambda_n^{(2)}=\lambda_n^{(3)}=\cdots =\lambda_n^{(n)}$, and the center if and only if $\lambda_n^{(1)}=\lambda_n^{(2)}=\cdots =\lambda_n^{(n)}$ $ \Leftrightarrow s=0 $.
\end{lemma}

Now we present main results of our paper.
\begin{theorem} \label{gcdeigenvalues}
Let $\lambda_n^{(1)} \leq \cdots \leq \lambda_n^{(n)}$ be the eigenvalues of the GCD matrix $(S_n)$ defined on $S_n =\{1,2, \ldots n \}$. Then we have
\begin{eqnarray}\label{gcdmineigenvaluebounds}
\frac{n(n+1)}{2}-s(n-1)^{1/2} < \lambda_n^{(1)} < \frac{n(n+1)}{2}-\bigg(\frac{ns^2 +2(n-1)}{n^2-n}\bigg)^{1/2}
\end{eqnarray}
and
\begin{eqnarray}\label{gcdmaxeigenvaluebounds}
\frac{n(n+1)}{2}+\bigg(\frac{ns^2 +2(n-1)}{n^2-n}\bigg)^{1/2} < \lambda_n^{(n)} < \frac{n(n+1)}{2} + s(n-1)^{1/2},
\end{eqnarray}
where
\begin{eqnarray} \label{sforgcd}
s=\bigg(\frac{2}{n} \sum_{i=1}^n (N^2 \ast \varphi) (i) -\frac{7n^2+12n+5}{12}\bigg)^{1/2}.
\end{eqnarray}

\end{theorem}

\begin{proof}
We first calculate $m$ and $s$ in Theorem~\ref{thmWolStyan} for $(S_n)$ in terms of arithmetical functions $N^2$  and $\varphi$, Euler's totient. It is clear that $m=n(n+1)/2$. Then we have
\begin{eqnarray*}
s^2 &=& \frac{1}{n} \sum_{i=1}^n \sum_{j=1}^n (i,j)^2 -\frac{1}{n^2} \bigg(\frac{n(n+1)}{2} \bigg)^2 \\
&=& \frac{1}{n} \bigg[ 2\sum_{i=1}^n \sum_{j=1}^i (i,j)^2 - \sum_{i=1}^n i^2 \bigg]-\frac{(n+1)^2}{4} \\ &=& \frac{2}{n} \sum_{i=1}^n (N^2 \ast \varphi)(i) - \frac{7n^2+12n+5}{12}.
\end{eqnarray*}
The last equality follows from Lemma~1 in \cite{bordolles2010} which was proved by Ces\'{a}ro \cite{cesaro1885} for the first time. From Theorem~\ref{thmWolStyan} we obtain the left-hand side of (\ref{gcdmineigenvaluebounds}) and the right-hand side of (\ref{gcdmaxeigenvaluebounds}).
On the other hand, we have
\begin{eqnarray*}
 n^2(m-\lambda_n^{(1)})^2 &=& n^2 \bigg( \frac{1}{n} \sum_{i=1}^n \lambda_n^{(i)} - \lambda_n^{(1)} \bigg)^2 \\
 &=& \bigg( \sum_{i=1}^n \lambda_n^{(i)} - n \lambda_n^{(1)} \bigg)^2 \\
 &=& \sum_{i=1}^n (\lambda_n^{(i)} -  \lambda_n^{(1)})^2 + \sum_{j\neq k} (\lambda_n^{(j)}-\lambda_n^{(1)})(\lambda_n^{(k)}-\lambda_n^{(1)})
\end{eqnarray*}
In the proof of Lemma~\ref{lemma2.2.WolStyan} (Lemma~2.2 in \cite{WolkowiczStyan1980}) the second sum is omitted but we improve the upper bound for $\lambda_n^{(1)}$ and the lower bound for $\lambda_n^{(n)}$ of $(S_n)$ in Lemma~\ref{lemma2.2.WolStyan}. From the well-known fact that $\lambda_{min}(A)<a_{ii}<\lambda_{max}(A)$ for an $n \times n$ Hermitian matrix $A=(a_{ij})$ we have $\lambda_n^{(n)}-\lambda_n^{(1)}>n-1$ for all $n\geq2$ and also from Cauchy's interlacing inequalities \cite{HornJohnson} we have that $\lambda_n^{(n-1)}-\lambda_n^{(1)}>\lambda_3^{(2)}-\lambda_3^{(1)}$ for all $n\geq3$. By a simple calculation in Maple $\lambda_3^{(2)}\cong1,460$ and $\lambda_3^{(1)} \cong 0,324$. Thus we have
\begin{eqnarray*}
\sum_{j\neq k} (\lambda_n^{(j)}-\lambda_n^{(1)})(\lambda_n^{(k)}-\lambda_n^{(1)}) > 2(n-1).
\end{eqnarray*}
Finally we have
\begin{eqnarray*}
\lambda_n^{(1)} \geq \frac{n(n+1)}{2}- \bigg( \frac{ns^2+2(n-1)}{n^2-n} \bigg)^{1/2}.
\end{eqnarray*}
Similarly expanding $n^2(\lambda_n^{(n)}-m)^2$ we obtain the left-hand inequality in (\ref{gcdmaxeigenvaluebounds}). This completes the proof.
\end{proof}

It should be noted that the upper (lower) bound for the smallest (largest) eigenvalue $\lambda_n^{(1)}$ ($\lambda_n^{(n)}$) in (\ref{gcdmineigenvaluebounds}) (in (\ref{gcdmaxeigenvaluebounds})) is better than the one in the right (left) hand side of (\ref{min}) ( (\ref{max}) ) for the GCD matrix $(S_n)$.

\begin{theorem} \label{lcmeigenvalues}
Let $\mu_n^{(1)} \leq \cdots \leq \mu_n^{(n)}$ be the eigenvalues of the LCM matrix $[S_n]$ on $S_n =\{1,2, \ldots n \}$. Then we have
\begin{eqnarray}\label{lcmmineigenvaluebounds}
\frac{n(n+1)}{2}-s(n-1)^{1/2} < \mu_n^{(1)} < \frac{n(n+1)}{2}-\bigg(\frac{s^2 +32(n-1)}{n-1}\bigg)^{1/2}
\end{eqnarray}
and
\begin{eqnarray}\label{lcmmaxeigenvaluebounds}
\frac{n(n+1)}{2}+\bigg(\frac{s^2 +32(n-1)}{n-1}\bigg)^{1/2} < \mu_n^{(n)} < \frac{n(n+1)}{2} + s(n-1)^{1/2},
\end{eqnarray}
where
\begin{eqnarray} \label{sforlcm}
s=\bigg[ \frac{1}{3n} \sum_{i=1}^n  i^2 \bigg( N\big( N\mu \ast (N+1)(2N+1)\big) \ast \zeta \bigg)(i) -\frac{7n^2+12n+5}{12}\bigg]^{1/2}.
\end{eqnarray}

\end{theorem}

\begin{proof}
We first calculate $m$ and $s$ in Theorem~\ref{thmWolStyan} for $[S_n]$ in terms of arithmetical functions $N$, $\mu$ and $\zeta$ by a similar method in \cite{Bordolles2007}. Again $m=n(n+1)/2$. Then we have
\begin{eqnarray*}
s^2 &=& \frac{1}{n} \sum_{i=1}^n \sum_{j=1}^n [i,j]^2 -\frac{1}{n^2} \bigg(\sum_{i=1}^n i \bigg)^2 \\
&=& \frac{1}{n} \bigg[ 2\sum_{i=1}^n \sum_{j=1}^i [i,j]^2 - \sum_{i=1}^n i^2 \bigg]-\frac{(n+1)^2}{4} \\ &=& \frac{2}{n} \sum_{i=1}^n \sum_{j=1}^i [i,j]^2 - \frac{7n^2+12n+5}{12}.
\end{eqnarray*}
Now we calculate the sum of squares of lcms
\begin{eqnarray*}
 \sum_{j=1}^i [i,j]^2 &=& \sum_{j=1}^i i^2 \frac{j^2}{(i,j)^2}\\
 &=& i^2 \sum_{d|i} \frac{1}{d^2}  \sum_{j=1 \atop (n,j)=d}^i \frac{1}{j^2} \\
 &=& i^2 \sum_{d|i} \sum_{k\leq i/d \atop (k,i/d)=1} k^2.
\end{eqnarray*}
Here we have for every integer $t$
\begin{eqnarray*}
 \sum_{k\leq t \atop (k,t)=1 } k^2 &=& \sum_{d|t} d^2\mu(d) \sum_{m \leq t/d} m^2 \\
 &=& \sum_{d|t} d^2\mu(d) \frac{t/d (t/d+1)(2t/d+1)}{6} \\
 &=& \frac{t}{6} \sum_{d|t} d \mu(d) (t/d+1)(2t/d+1) \\
 &=& \bigg[ \frac{N}{6} \big( (N \mu) \ast ((N+1)(2N+1)) \big) \bigg] (t).
\end{eqnarray*}
Thus we obtain (\ref{sforlcm}).

Again we reproof Lemma~\ref{lemma2.2.WolStyan} (Lemma~2.2 in \cite{WolkowiczStyan1980}) and hence we improve the inequalities in Lemma~\ref{lemma2.2.WolStyan} for $\mu_n^{(1)}$ and $\mu_n^{(n)}$ of $[S_n]$. Since $\mu_{max}(A)-\mu_{min}(A) \geq 2\max_{i\neq j}|a_{ij}|$ for an $n \times n$ Hermitian matrix $A=(a_{ij})$ (see \cite{JohnsonKumarWolkowicz1985}) and also it is clear that $[n,n-1]=n(n-1)$ for every positive integer $n$, we have $\mu_n^{(n)}-\mu_n^{(1)}>2n(n-1)$ for all $n\geq2$. By Cauchy's interlacing inequalities \cite{HornJohnson} we have $\mu_n^{(n-1)}-\mu_n^{(1)} > \mu_4^{(3)}-\mu_4^{(1)}$ for all $n\geq3$. By a simple calculation in Maple $\mu_4^{(3)}\cong-0,312$ and $\mu_4^{(1)} \cong -8,843$. Thus we have
\begin{eqnarray*}
\sum_{j\neq k} (\mu_n^{(j)}-\mu_n^{(1)})(\mu_n^{(k)}-\mu_n^{(1)}) > 32n(n-1).
\end{eqnarray*}
Finally we have
\begin{eqnarray*}
\mu_n^{(1)} \geq m- \bigg( \frac{s^2+32(n-1)}{n-1} \bigg)^{1/2}.
\end{eqnarray*}
Similarly expanding $n^2(\mu_n^{(n)}-m)^2$ we obtain the left-hand inequality in (\ref{gcdmaxeigenvaluebounds}).

\end{proof}
Again it should be noted that the upper (lower) bound for the smallest (largest) eigenvalue $\lambda_n^{(1)}$ ($\lambda_n^{(n)}$) in (\ref{lcmmineigenvaluebounds}) (in (\ref{lcmmaxeigenvaluebounds})) is better than the one in the right (left) hand side of (\ref{min}) ( (\ref{max}) ) for the LCM matrix $(S_n)$.

\section{Comments and An Open Problem}
In the study of eigenvalues of GCD, LCM and related matrices, some authors \cite{Altinisik2009, Hong2008, HongEnochLee, Hong, HongLoewy2011} investigated asymptotic behavior of the eigenvalues of such matrices and also they gave lower bounds for the smallest eigenvalue and upper bounds for the largest eigenvalue of these matrices. On the other hand, some other authors \cite{IlmonenHaukkanenMerikoski, Mattila2014, Mattila2012, Mattila2012paper, MattilaHauk2014} obtained such lower and upper bounds for the eigenvalues of these matrices by using matrix theoretic techniques. In these papers, the authors obtained lower bounds for the smallest eigenvalue of GCD-related matrices depending on some certain constants on the eigenvalues of particular matrices and some values of particular arithmetical functions. For example, now we emphasize a result of Hong, Theorem~2.2 in \cite{Hong2008}. Let $K_n$ be the set of all $n \times n$ lower triangular matrices that satisfy: Every main diagonal entry is 1, and every off-diagonal entry is 0 or 1. Also let $L_n=\{ YY^T: Y \in K_n \}$ and
$$
c_n = \min_{ Z \in L_n } \{ \mu_n ^{(1)} (Z) : \mu_n ^{(1)} (Z) \text{ is the smallest eigenvalue of } Z \}.
$$
Let $\lambda_n ^{(1)}$ be the smallest eigenvalue of the $n \times n$ matrix $(f(x_i ,x_j))$ defined on $S=\{x_1, x_2,\ldots ,x_n\}$, where $f$ be an arithmetical function in
$$
C_S= \{ f: (f\ast \mu)(d) >0 \text{ whenever } d|x \text{ for any } x \in S \}.
$$
Then Hong \cite{Hong2008} proved that $\lambda_n ^{(1)} \geq c_n \cdot \min_{1 \leq i \leq n} \{ (f \ast \mu ) (x_i ) \}.$

In this paper, we have obtained not only a lower (an upper) bound but also an upper (a lower) bound for the smallest (largest) eigenvalue of the GCD matrix as well as the LCM matrix by using a different technique from above papers. Indeed, our bounds depend on particular arithmetical functions and the size $n$ of our matrices, not such a constant $c_n$. In order to do this, we improve the results of Wolkowicz and Styan \cite{WolkowiczStyan1980} for the GCD matrix and the LCM matrix.

In this context, Mattilla and Haukkanen \cite{Mattila2012,Mattila2012paper} proved that every eigenvalue of the matrix $A^{\alpha,\beta}_n$ lies in the real interval
\begin{eqnarray*}
\bigg[2 \min\{1, n^{\alpha+\beta} \} - T_n \max\{1, n^{2\beta}\} \max_{1 \leq i \leq n } |J_{\alpha - \beta (i)}| , T_n \max\{1, n^{2\beta}\} \max_{1 \leq i \leq n } |J_{\alpha - \beta (i)}| \bigg].
\end{eqnarray*}
Here $t_n$ and $T_n$ are the smallest and the largest eigenvalues of $EE^T$, where $E$ is the $n \times n$ matrix $ (e_{ij})$ whose the $ij-$entry is 1 if $j|i$ and $0$ otherwise. These bounds provided by above interval are valid for all values of $\alpha$ and $\beta$ but it is natural that these bounds are not good enough for particular values of $\alpha$ and $\beta$. For example, $A^{1,0}_n$ is the GCD matrix $(S_n)$ and above interval is
$$\big[ 2-T_n\cdot \max_{1 \leq i \leq n } |\varphi (i)|, T_n\cdot \max_{1 \leq i \leq n } | \varphi (i)| \big].
$$
For $n=20$ the interval approximately is $[-595.8214, 597.8214]$. On the other hand, from Theorem~\ref{gcdeigenvalues} we have approximately -40.2114 and 7.8123 as a lower and an upper bound for the smallest eigenvalue of $(S_{20})$ and also we have approximately 13.1876 and 61.2114 as a lower and an upper bound for the largest eigenvalue of $(S_{20})$.

On the other hand, in this paper we study the eigenvalues of the LCM matrix for the first time. We know that the GCD matrix $(S)$ defined on any set $S=\{x_1,x_2,\ldots,x_n\}$ is positive definite. Contrary to the GCD matrix, the LCM matrix $[S]$ defined on $S$ need not be positive definite. For example, the eigenvalues of $[S_2]$ are exactly
$\lambda_1= \frac{3}{2}-\frac{1}{2} \sqrt{17}$ and $\lambda_2=
\frac{3}{2}+\frac{1}{2} \sqrt{17}$. Let $n \geq2$. Then by Cauchy's interlacing
inequalities (see \cite{HornJohnson}) at least one eigenvalue of
$[S]=([x_i,x_j])$ must be negative and at least one eigenvalue of
$[S]=([x_i,x_j])$ must be positive. Thus the LCM matrix $[S]=([x_i,x_j])$
defined on $S=\{x_1,x_2,\ldots,x_n\}$ can not be positive definite or negative
definite. Then we conclude this observation with an open problem.

\textbf{Problem.} How many eigenvalues of $[S_n]=([i,j])$ defined on $S_n=\{1,2, \ldots, n \}$ are
positive? More generally, how many eigenvalues of $[S]=([x_i,x_j])$ defined on
$S=\{x_1,x_2,\ldots,x_n\}$ are positive?

% ----------------------------------------------------------------
\bibliographystyle{amsplain}
\bibliography{}

\begin{thebibliography}{99}


\bibitem{Altinisik2005} E. Alt\i n\i \c{s}\i k, B. E. Sagan and N. Tu\~{g}lu, GCD matrices, posets, and
nonintersecting paths, Linear and Multilinear Algebra 53(2) (2005) 75-84.

\bibitem{Altinisik2009} E. Alt\i n\i \c{s}\i k, On inverses of GCD matrices associated with multiplicative functions and a proof of the Hong-Loewy conjecture, Linear Algebra Appl. 430 (2009) 1313-1327.

\bibitem{BeslinLigh} S. Beslin and S. Ligh, Greatest common divisor matrices,
Linear Algebra Appl. 118 (1989) 69-76.

\bibitem{Bordolles2007} O. Bordellès, Mean values of generalized gcd-sum and lcm-sum functions, J. Integer Seq. 10 (2007), Article 07.9.2, 13 pp.

\bibitem{bordolles2010} O. Bordellès, The composition of gcd and certain arithmetic function, J. Integer Seq. 13 (2010), no. 3, Article 10.7.1, 22 pp.


\bibitem{BourqueLigh1992} K. Bourque and S. Ligh, On GCD and LCM matrices, Linear Algebra Appl. 174 (1992) 65-74.

\bibitem{cesaro1885} E. Ces\'{a}ro, \'{E}tude moyenne du plus grand commun diviseur de deux nombres, Ann. Mat. Pura Appl. 13 (1885), 235–250.

\bibitem{Hauk1997} P. Haukkanen, J. Wang and J. Sillanp\"{a}\"{a}, On Smith's determinant,
Linear Algebra Appl. 258 (1997) 251-269.

\bibitem{Hong2008} S. Hong, Asymptotic behavior of largest eigenvalue of matrices associated with completely even functions (mod $r$) Asian-Europ.J. Math. 1 (2008) 225-235.

\bibitem{HongEnochLee} S. Hong and K. S. Enoch Lee, Asymptotic behavior of eigenvalues of reciprocal power LCM matrices, Glasg. Math. J.
50 (2008) 163-174.

\bibitem{Hong} S. Hong and R. Loewy, Asymptotic behavior of eigenvalues of
greatest common divisor matrices, Glasg. Math. J. 46 (2004) 303-308.

\bibitem{HongLoewy2011} S. Hong and R. Loewy, Asymptotic behavior of the smallest eigenvalue of matrices associated
with completely even functions (modr), Int. J. Number Theory 7 (2011)
1681-1704.

\bibitem{HornJohnson} R. Horn, C. R. Johnson, Matrix
Analysis, Cambridge University Press, Cambridge, London, 1985.

\bibitem{IlmonenHaukkanenMerikoski} P. Ilmonen, P. Haukkanen and J. K. Merikoski, On eigenvalues of meet and join matrices associated with incidence functions, Linear Algebra Appl. 429 (2008) 859-874.

\bibitem{JohnsonKumarWolkowicz1985} C. R. Johnson, R. Kumar, H. Wolkowicz, Lower bounds for the spread of a matrix, Linear Algebra Appl. 71 (1985) 161-173.

\bibitem{KorkeeHauk2003} I. Korkee and P. Haukkanen, On meet and join matrices associated with incidence functions, Linear Algebra Appl. 372 (2003) 127-153.

\bibitem{Lindqvist} P. Lindqvist and K. Seip, Note on some greatest common divisor matrices,
Acta Arith. 84.2 (1998) 149-154.

\bibitem{Mattila2014} M. Mattila, On the eigenvalues of combined meet and join matrices, arXiv:1403.5907v1.

\bibitem{Mattila2012} M. Mattila, P. Haukkanen, On the eigenvalues of certain number-theoretic matrices, International Conference in Number Theory and Applications 2012.

\bibitem{Mattila2012paper} M. Mattila, P. Haukkanen, On the eigenvalues of certain number-theoretic matrices. East-West J. Math. 14 (2012), no. 2, 121-130.

\bibitem{MattilaHauk2014} M. Mattila and P. Haukkanen, On the positive definiteness and eigenvalues of meet and join matrices, Discrete Math. 326 (2014) 9-19.

\bibitem{SandorCrs} J. S\'{a}ndor and B. Crstici,  Handbook of Number Theory vol.II, Springer Verlag, New York, 2005.

\bibitem{Siva} R. Sivaramakrishnan, Classical Theory of Arithmetic Functions, Marcel Deccer, New York, 1989.

\bibitem{Smith} H. J. S. Smith, On the value of a certain artihmetical determinant,
Proc. London Math. Soc. Ser.1 7 (1876) 208-212.

\bibitem{Wintner} A. Wintner, Diophantine approximations and Hilbert's space, Amer. J.
Math. 66 (1944) 564-578.

\bibitem{WolkowiczStyan1980} H. Wolkowicz and G. P. H. Styan, Bounds for eigenvalues using traces, Linear Algebra Appl. 29 (1980) 471-506.

\end{thebibliography}

\end{document}